\documentclass[12pt,a4paper]{amsart}
\usepackage{etex} 
\usepackage[latin1]{inputenc}
\usepackage[T1]{fontenc}
\usepackage[right=2cm,left=2cm,top=3cm,bottom=3cm]{geometry}
\usepackage{ifthen}
\usepackage{enumerate}
\usepackage{amsmath}
\usepackage{amsthm}
\usepackage{amsfonts}
\usepackage{amssymb} 
\usepackage{color,colortbl} 
\usepackage{latexsym}
\usepackage[np]{numprint}
\npdecimalsign{\ensuremath{.}}

\title{On certain sums of number theory}
\author{Olivier Bordellès}
\address{2 allée de la combe \\ 43000 Aiguilhe \\ France}
\email{borde43@wanadoo.fr}
\date{}
\dedicatory{}

\newcommand{\Z}{\mathbb {Z}}

\newcommand{\R}{\mathbb {R}}
\newcommand{\C}{\mathbb {C}}

\newtheorem{theorem}{Theorem}[section]
\newtheorem{prop}[theorem]{Proposition}
\newtheorem{coro}[theorem]{Corollary}
\newtheorem{lem}[theorem]{Lemma}

\theoremstyle{remark}

\makeatletter
\@namedef{subjclassname@2020}{%
  \textup{2020} Mathematics Subject Classification}
\makeatother

\begin{document}

\begin{abstract}
We study sums of the shape $\sum_{n \leqslant x} f \left( \lfloor x/n \rfloor \right)$ where $f$ is either the von Mangoldt function or the Dirichlet-Piltz divisor functions. We improve previous estimates when $f = \Lambda$ and $f = \tau$, and provide new results when $f = \tau_r$ with $r \geqslant 3$, breaking the $\frac{1}{2}$-barrier in each case. The functions $f=\mu^2$, $f=2^\omega$ and $f=\omega$ are also investigated.
\end{abstract}


\maketitle

\thispagestyle{myheadings}
\font\rms=cmr8 
\font\its=cmti8 
\font\bfs=cmbx8

\section{Introduction and results}
\label{s1}

\noindent
Recently, there has been a great deal of interest in estimating sums of the form
$$\sum_{n \leqslant x} f \left( \left \lfloor \frac{x}{n} \right \rfloor \right)$$
where $\lfloor x \rfloor$ is the integer part of $x \in \R$, and $f$ is an arithmetic function. Historically, the first one goes back to Dirichlet in the middle of the 19th century when he proved that
$$\sum_{n \leqslant x} \left \lfloor \frac{x}{n} \right \rfloor = x \log x + x (2 \gamma - 1) + O \left( \sqrt{x} \right).$$
Subsequently, the exponent in the error term has been improved, the best result to date being $x^{517/\np{1648} + \varepsilon}$, which is due to Bourgain \& Watt \cite{bouwa}. In \cite{bor}, the authors established a quite general result involving arithmetic functions $f$ which are not too large. More precisely, if $f$ satisfies
$$\sum_{n \leqslant x} \left| f(n) \right|^2 \ll x^\alpha$$
for some $\alpha \in \left( 0,2  \right)$, then it is proved that
$$\sum_{n \leqslant x} f \left( \left \lfloor \frac{x}{n} \right \rfloor \right) = x \sum_{n=1}^\infty \frac{f(n)}{n(n+1)} + O \left( x^{\frac{1}{3} (\alpha+1)} (\log x)^{\frac{1}{3} (\alpha+1) + o(1)} \right).$$
This estimate was then improved independently by Wu \cite[Theorem~1.2]{wu} and Zhai \cite[Theorem~1]{zhai} who proved that
$$\sum_{n \leqslant x} f \left( \left \lfloor \frac{x}{n} \right \rfloor \right) = x \sum_{n=1}^\infty \frac{f(n)}{n(n+1)} + O \left( x^{\frac{1}{2} (\alpha+1)} (\log x)^{\theta} \right)$$
provided that $f(n) \ll n^\alpha (\log n)^\theta$ for some $\alpha \in \left[ 0,1 \right)$ and $\theta \geqslant 0$. For arithmetic functions $f$ satisfying the Ramanujan hypothesis $f(n) \ll n^\varepsilon$, this implies
\begin{equation}
   \sum_{n \leqslant x} f \left( \left \lfloor \frac{x}{n} \right \rfloor \right) = x \sum_{n=1}^\infty \frac{f(n)}{n(n+1)} + O \left( x^{\frac{1}{2} + \varepsilon} \right). \label{eq:1/2-barrier}
\end{equation}
The question of breaking the $\frac{1}{2}$-barrier for specific arithmetic functions $f$ then arises naturally. Using Vaughan's identity and the exponent pair $\left (\frac{1}{6},\frac{2}{3} \right)$, Ma and Wu \cite{mawu} showed that
\begin{equation}
   \sum_{n \leqslant x} \Lambda \left( \left \lfloor \frac{x}{n} \right \rfloor \right) = x \sum_{n=1}^\infty \frac{\Lambda(n)}{n(n+1)} + O \left( x^{\frac{35}{71} + \varepsilon} \right). \label{eq:Ma_Wu}
\end{equation}
In a similar but simpler way, Ma and Sun \cite{masun} proved that
\begin{equation}
   \sum_{n \leqslant x} \tau \left( \left \lfloor \frac{x}{n} \right \rfloor \right) = x \sum_{n=1}^\infty \frac{\tau(n)}{n(n+1)} + O \left( x^{\frac{11}{23} + \varepsilon} \right). \label{eq:Ma_Sun}
\end{equation}
Note that $\frac{35}{71} \approx \np{0.4929} \dotsc$ and $\frac{11}{23} \approx \np{0.4782} \dotsc$ The aim of this work is to improve these results when $f=\Lambda$ and $f=\tau$, to extend them to the case $f = \tau_r$ for some fixed integer $r \geqslant 2$, and also to study the cases $f=\mu^2$, $f=2^\omega$ and $f=\omega$. 

\begin{theorem}
\label{th:Lambda}
Let $(k,\ell)$ be an exponent pair satisfying $k \leqslant \frac{1}{6}$, $3k+4\ell \geqslant 1$ and $\ell^2 + \ell + 3 - k(5-\ell) - 9k^2 >0$. For any $\varepsilon > 0$ and $x$ sufficiently large, we have
$$\sum_{n \leqslant x} \Lambda \left( \left \lfloor \frac{x}{n} \right \rfloor \right) = x \sum_{n=1}^\infty \frac{\Lambda(n)}{n(n+1)} + O_\varepsilon \left( x^{\frac{14(k+1)}{29k-\ell+30} + \varepsilon} \right).$$
\end{theorem}

\begin{theorem}
\label{th:Dirichlet-Piltz}
Let $r \geqslant 2$ be any fixed integer and $(k,\ell)$ be an exponent pair satisfying
\begin{equation}
   1-\ell > k(r-1). \label{eq:condition_tau_r}
\end{equation}
For any $\varepsilon > 0$ and $x$ sufficiently large, we have
$$\sum_{n \leqslant x} \tau_r \left( \left \lfloor \frac{x}{n} \right \rfloor \right) = x \sum_{n=1}^\infty \frac{\tau_r(n)}{n(n+1)} + O_{\varepsilon,r} \left( x^{\frac{k(r-1)+\ell+r-1}{k(r-1)+\ell+2r-1} + \varepsilon} \right).$$
\end{theorem}

\noindent
It is proved in \cite[Theorem~6]{bou} that $\left( \frac{13}{84} + \varepsilon, \frac{55}{84} + \varepsilon \right)$ is an exponent pair. We use this result in the cases $f=\Lambda$ and $f = \tau$, and the exponent pair $A \left( \frac{13}{84} + \varepsilon , \frac{55}{84} + \varepsilon \right) = \left( \frac{13}{194} + \varepsilon , \frac{76}{97} + \varepsilon \right)$ when $f = \tau_3$. For $r \geqslant 4$, the condition \eqref{eq:condition_tau_r} requires having $k$ very small. Recently, some improvements in exponential sums have appeared in the literature. As an application, Heath-Brown \cite[Theorem~2]{hea} proved that, for all $m \in \Z_{\geqslant 3}$
$$(k,\ell) = \left( \frac{2}{(m-1)^2(m+2)}\, , \, 1 - \frac{3m-2}{m(m-1)(m+2)} + \varepsilon \right)$$
is an exponent pair. For the function $\tau_r$ with $r \geqslant 4$, we use this result with $m=2r-1$. Putting altogether, we derive the next estimates.

\begin{coro}
Let $r \geqslant 4$ be any fixed integer. For any $\varepsilon > 0$ and $x \geqslant e$ sufficiently large, we have
\begin{align*}
   & \sum_{n \leqslant x} \Lambda \left( \left \lfloor \frac{x}{n} \right \rfloor \right) = x \sum_{n=1}^\infty \frac{\Lambda(n)}{n(n+1)} + O_\varepsilon \left( x^{\frac{97}{203} + \varepsilon} \right) \, ; \\
   & \sum_{n \leqslant x} \tau \left( \left \lfloor \frac{x}{n} \right \rfloor \right) = x \sum_{n=1}^\infty \frac{\tau(n)}{n(n+1)} + O_\varepsilon \left( x^{\frac{19}{40} + \varepsilon} \right) \, ; \\
   & \sum_{n \leqslant x} \tau_3 \left( \left \lfloor \frac{x}{n} \right \rfloor \right) = x \sum_{n=1}^\infty \frac{\tau_3(n)}{n(n+1)} + O_\varepsilon \left( x^{\frac{283}{574} + \varepsilon} \right) \, ; \\
   & \sum_{n \leqslant x} \tau_r \left( \left \lfloor \frac{x}{n} \right \rfloor \right) = x \sum_{n=1}^\infty \frac{\tau_r(n)}{n(n+1)} + O_\varepsilon \left( x^{\frac{1}{2}-\frac{1}{2(4r^3-r-1)} + \varepsilon} \right). \\
\end{align*}
\end{coro}

\noindent
Note that $\frac{97}{203} \approx 0.4778$, $\frac{19}{40} = \np{0.475}$, $\frac{283}{574} \approx \np{0.493}$ and 

\begin{footnotesize}
\begin{center}
\begin{tabular}{lccc}
$r$ & $4$ & $5$ & $6$ \\
& & & \\
$\frac{1}{2}-\frac{1}{2(4r^3-r-1)}$ & $\frac{125}{251}$ & $\frac{493}{988}$ & $\frac{428}{857}$ \\
& & & \\
Approx & $\np{0.498}$ & $\np{0.499}$ & $\np{0.4994}$ 
\end{tabular}
\end{center}
\end{footnotesize}

\noindent
For the functions $\mu_2 = \mu^2$ and $2^\omega$, we have the following estimates.

\begin{theorem}
\label{th:squarefree}
For any $\varepsilon > 0$ and $x \geqslant e$ sufficiently large, we have
$$\sum_{n \leqslant x} \mu_2 \left( \left \lfloor \frac{x}{n} \right \rfloor \right) = x \sum_{n=1}^\infty \frac{\mu_2(n)}{n(n+1)} + O_\varepsilon \left( x^{\frac{\np{1919}}{\np{4268}} + \varepsilon} \right).$$
\end{theorem}

\begin{theorem}
\label{th:unitary_divisors}
Let $(k,\ell)$ be an exponent pair such that $k + \ell < 1$. For any $\varepsilon > 0$ and $x \geqslant e$ sufficiently large, we have
$$\sum_{n \leqslant x} 2^{\omega\left( \left \lfloor x/n \right \rfloor \right)} = x \sum_{n=1}^\infty \frac{2^{\omega(n)}}{n(n+1)} + O_\varepsilon \left( x^{\frac{2(k+1)}{3k-\ell+5} + \varepsilon} \right).$$
In particular
$$\sum_{n \leqslant x} 2^{\omega\left( \left \lfloor x/n \right \rfloor \right)} = x \sum_{n=1}^\infty \frac{2^{\omega(n)}}{n(n+1)} + O_\varepsilon \left( x^{\frac{97}{202} + \varepsilon} \right).$$
\end{theorem}

\noindent
Note that $\frac{\np{1919}}{\np{4268}} \approx \np{0.4496}$ and $\frac{97}{202} \approx \np{0.4802}$.

\medskip

\noindent
Our last estimate deals with the additive function $\omega$ and improves the main result of \cite{bouka}.

\begin{theorem}
\label{th:omega}
For any $\varepsilon > 0$ and $x \geqslant e$ sufficiently large, we have
$$\sum_{n \leqslant x} \omega \left( \left \lfloor \frac{x}{n} \right \rfloor \right) = x \sum_{n=1}^\infty \frac{\omega(n)}{n(n+1)} + O_\varepsilon \left( x^{\frac{455}{914} + \varepsilon} \right).$$
\end{theorem}

\noindent
Note that $\frac{455}{914} \approx \np{0.4978}$. Also note that the error term can be sharpened to $O_\varepsilon \left( x^{\np{0.4958} + \varepsilon} \right)$ if we choose the exponent pair $\left( \frac{13}{194} + \varepsilon , \frac{76}{97} + \varepsilon \right)$ instead of $\left( \frac{1}{6}, \frac{2}{3} \right)$.

\section{Notation}

\noindent
If $f$ and $g$ are any arithmetic functions, $f \star g$ is the Dirichlet convolution product defined by
$$(f \star g)(n) = \sum_{d \mid n} f(d) g(n/d).$$
Let $\mu$ be the M\"{o}bius function, $\Lambda = \mu \star \log $ is the von Mangoldt function, and $\mu_2 = \mu^2$ is the characteristic function of the set of squarefree numbers. As usual, $\omega(n)$ is the number of distinct prime factors of $n$ with the convention $\omega(1)=0$, so that $2^{\omega(n)}$ counts the number of unitary divisors of $n$. If $r \geqslant 1$ is any fixed positive integer, the Dirichlet-Piltz divisor function $\tau_r$ is inductively defined by $\tau_1 = \mathbf{1}$ and, for $r \geqslant 2$, $\tau_r = \tau_{r-1} \star \mathbf{1}$, and it is customary to set $\tau = \tau_2$. Finally, for any $x \in \R$, $e(x) = e^{2 i \pi x}$ and $\psi(x) = x - \lfloor x \rfloor - \frac{1}{2}$ is the $1$st Bernoulli function.

\section{Preliminary}

\noindent
The next result relates our problem to estimating certain exponential sums.

\begin{prop}
\label{pro:preliminary}
Let $x \geqslant e$ large, $f:\Z_{\geqslant 1} \to \C$ satisfying $f(n) \ll n^\varepsilon$ and let $x^{1/3} \leqslant N < x^{1/2}$ be a parameter. Then, for all $H \in \Z_{\geqslant 1}$
\begin{multline*}
   \sum_{n \leqslant x} f \left( \left \lfloor \frac{x}{n} \right \rfloor \right) = x \sum_{n=1}^\infty \frac{f(n)}{n(n+1)} \\
   + O \left\lbrace Nx^\varepsilon + x^\varepsilon \max_{N < D \leqslant x/N} \left( \frac{D}{H} + \sum_{h \leqslant H} \frac{1}{h} \sum_{a=0}^1 \left| \sum_{D < d \leqslant 2D} f(d) \, e \left( \frac{hx}{d+a}\right) \right | \right) \right\rbrace.
\end{multline*}   
\end{prop}

\begin{proof}
Note first that the series in the main term above converges absolutely. Following \cite{masun,mawu}, we split the sum into two subsums
$$\sum_{n \leqslant x} f \left( \left \lfloor \frac{x}{n} \right \rfloor \right) = \left( \sum_{n \leqslant N} + \sum_{N < n \leqslant x} \right) f \left( \left \lfloor \frac{x}{n} \right \rfloor \right) := S_1 + S_2.$$
where $x^{1/3} \leqslant N < x^{1/2}$ is a parameter at our disposal. Trivially
$$S_1 \ll x^\varepsilon \sum_{n \leqslant N} \frac{1}{n^\varepsilon} \ll Nx^\varepsilon.$$
Next
\begin{align*}
   S_2 &= \sum_{d \leqslant x/N} f(d) \left( \left \lfloor \frac{x}{d} \right \rfloor - \left \lfloor \frac{x}{d+1} \right \rfloor \right) + O \left\lbrace x^\varepsilon \left( 1 + xN^{-2}\right) \right\rbrace \\
   &= \sum_{d \leqslant x/N} f(d) \left( \frac{x}{d(d+1)} - \psi \left(  \frac{x}{d} \right) + \psi \left( \frac{x}{d+1} \right) \right) + O \left( x^{1+\varepsilon} N^{-2}\right) \\
   &= x \sum_{d=1}^\infty \frac{f(d)}{d(d+1)} - x \sum_{d> x/N}^\infty \frac{f(d)}{d(d+1)} + \sum_{d \leqslant N} f(d) \left( \psi \left( \frac{x}{d+1} \right) - \psi \left(  \frac{x}{d} \right) \right) \\
   & \qquad + \sum_{N < d \leqslant x/N} f(d) \left( \psi \left( \frac{x}{d+1} \right) - \psi \left(  \frac{x}{d} \right) \right) + O \left( x^{1+\varepsilon} N^{-2}\right).
\end{align*}
Now the condition $f(n) \ll n^\varepsilon$ entails that
$$\left| \sum_{d \leqslant N} f(d) \left( \psi \left( \frac{x}{d+1} \right) - \psi \left(  \frac{x}{d} \right) \right) \right| \leqslant \sum_{d \leqslant N} \left| f(d) \right| \ll N^{1+\varepsilon}$$
and, by partial summation
$$\sum_{d> x/N}^\infty \frac{f(d)}{d(d+1)} \ll \left( \frac{x}{N}\right)^{\varepsilon - 1}.$$
Therefore
$$S_2 = x \sum_{d=1}^\infty \frac{f(d)}{d(d+1)} + \sum_{N < d \leqslant x/N} f(d) \left( \psi \left( \frac{x}{d+1} \right) - \psi \left(  \frac{x}{d} \right) \right) + O \left( N x^\varepsilon + x^{1+\varepsilon} N^{-2} \right) $$
and note that $xN^{-2} \leqslant N$ since $N \geqslant x^{1/3}$. Hence
$$\sum_{n \leqslant x} f \left( \left \lfloor \frac{x}{n} \right \rfloor \right) = x \sum_{d=1}^\infty \frac{f(d)}{d(d+1)} + \sum_{N < d \leqslant x/N} f(d) \left( \psi \left( \frac{x}{d+1} \right) - \psi \left(  \frac{x}{d} \right) \right) + O \left( N x^\varepsilon \right).$$
We complete the proof with the usual Vaaler's approximation of the function $\psi$ by trigonometric polynomials \cite{vaa}, implying the asserted result.
\end{proof}

\section{Useful decompositions}

\noindent
The next result is Vaughan's identity \cite{vau} or \cite[Chapter~24]{dav}. We use the functions
$$\mathbf{1}_U^- (n) = \begin{cases} 1, & \textrm{if\ } n \leqslant U \, ; \\ 0, & \textrm{otherwise}  \, ;\end{cases} \quad \textrm{and} \quad \mathbf{1}_U^+ (n) = \begin{cases} 1, & \textrm{if\ } n > U \, ; \\ 0, & \textrm{otherwise}. \end{cases}$$

\begin{prop}
\label{pro:Vaughan_identity}
Let $1 < R < R_1 \leqslant 2R$ and let $F : \left[ 1,\infty \right) \to \left[ 0,\infty \right)$ be any map. For all $1 \leqslant U \leqslant R^{1/2}$
\begin{multline*}
   \sum_{R < n \leqslant R_1} \Lambda(n) \, e \left( F(n) \right)  = \sum_{n \leqslant U} \mu(n) \sum_{\frac{R}{n} < m \leqslant \frac{R_1}{n}} \, e \left( F(mn) \right) \log m \\
   - \sum_{n \leqslant U^2} a_n \sum_{\frac{R}{n} < m \leqslant \frac{R_1}{n}} \, e \left( F(mn) \right) - \sum_{U < n \leqslant \frac{R_1}{U}} \Lambda(n) \sum_{\frac{R}{n} < m \leqslant \frac{R_1}{n}} b_m \, e \left( F(mn) \right)
\end{multline*}
with
$$a_n := \left( \mu \mathbf{1}_U^- \star \Lambda \mathbf{1}_U^- \right) (n) \quad \text{and} \quad b_m := \left( \mu \mathbf{1}_U^- \star \mathbf{1} \right) (m).$$
\end{prop}

\noindent
A similar result holds for the M\"{o}bius function.

\begin{prop}
\label{pro:Vaughan_identity_mu}
Let $1 < R < R_1 \leqslant 2R$ and let $F : \left[ 1,\infty \right) \to \left[ 0,\infty \right)$ be any map. For all $1 \leqslant U \leqslant R^{1/2}$
\begin{multline*}
   \sum_{R < n \leqslant R_1} \mu(n) \, e \left( F(n) \right)  = - \sum_{n \leqslant U} a_n \sum_{\frac{R}{n} < m \leqslant \frac{R_1}{n}} \, e \left( F(mn) \right) \log m \\
   - \sum_{U < n \leqslant U^2} a_n \sum_{\frac{R}{n} < m \leqslant \frac{R_1}{n}} \, e \left( F(mn) \right) \log m + \sum_{U < n \leqslant \frac{R_1}{U}} b_n \sum_{\max \left( U, \frac{R}{n} \right) < m \leqslant \frac{R_1}{n}} \mu(m) \, e \left( F(mn) \right)
\end{multline*}
with
$$a_n := \left( \mu \mathbf{1}_U^- \star \mu \mathbf{1}_U^- \right) (n) \quad \text{and} \quad b_n := \left( \mu \mathbf{1}_U^+ \star \mathbf{1} \right) (n).$$
\end{prop}

\noindent
The usual Dirichlet hyperbola principle, a proof of which can be found for instance in \cite[Theorem~2.4.1]{mur}, can be slightly extended to the following form. The proof is well-known.

\begin{lem}[Dirichlet hyperbola principle]
\label{le:Dirichlet_hyperbola}
Let $f,g: \Z_{\geqslant 1} \to \C$ be two arithmetic functions and $h : \left[ 1,\infty \right) \to \C$ be any map. For all $1 \leqslant U \leqslant x$
\begin{multline*}
   \sum_{n \leqslant x} \left( f \star g \right)(n) h(n) = \sum_{n \leqslant U} f(n) \sum_{m \leqslant x/n} g(m) h(mn) \\
   + \sum_{n \leqslant x/U} g(n) \sum_{m \leqslant x/n} f(m) h (mn) - \sum_{n \leqslant U} \sum_{m \leqslant x/U} f(n)g(m) h (mn).
\end{multline*}
\end{lem}

\noindent
Specifying $h(n) = e \left( F(n) \right)$ where $F : \left[ 1,\infty \right) \to \left[ 0,\infty \right)$ is any function, we immediately derive the next tool.

\begin{coro}
\label{cor:Dirichlet_exponential_principle}
Let $f,g: \Z_{\geqslant 1} \to \C$ be two arithmetic functions and $F : \left[ 1,\infty \right) \to \left[ 0,\infty \right)$ be any map. For all $R<R_1 \in \Z_{\geqslant 1}$ and $1 \leqslant U \leqslant R$
\begin{multline*}
   \sum_{R < n \leqslant R_1} \left( f \star g \right)(n)  \, e \left( F(n) \right) = \sum_{n \leqslant \frac{UR_1}{R}} f(n) \sum_{\frac{R}{n} < m \leqslant \frac{R_1}{n}} g(m) \, e \left( F(mn) \right) \\
   + \sum_{n \leqslant \frac{R}{U}} g(n) \sum_{\frac{R}{n} < m \leqslant \frac{R_1}{n}} f(m)  \, e \left( F(mn) \right) - \sum_{U < n \leqslant \frac{UR_1}{R}} f(n) \sum_{\frac{R}{n} < m \leqslant \frac{R}{U}} g(m)  \, e \left( F(mn) \right).
\end{multline*}
\end{coro}

\begin{proof}
By Lemma~\ref{le:Dirichlet_hyperbola} we first derive
\begin{align}
    \sum_{R < n \leqslant R_1} \left( f \star g \right)(n)  \, e \left( F(n) \right) &= \sum_{n \leqslant U} f(n) \sum_{\frac{R}{n} < m \leqslant \frac{R_1}{n}} g(m)  \, e \left( F(mn) \right) + \sum_{n \leqslant \frac{R}{U}} g(n) \sum_{\frac{R}{n} < m \leqslant \frac{R_1}{n}} f(m)  \, e \left( F(mn) \right) \notag \\
   & \qquad + \sum_{\frac{R}{U} < n \leqslant \frac{R_1}{U}} g(n) \sum_{m \leqslant \frac{R_1}{n}} f(m)  \, e \left( F(mn) \right) - \sum_{n \leqslant U} f(n)\sum_{\frac{R}{U} < m \leqslant \frac{R_1}{U}} g(m)  \, e \left( F(mn) \right) \label{eq:hyperbole_exp} \\
   &:= S_1 + S_2 + S_3 - S_4. \notag
\end{align} 
For $S_3$, interchanging the sums and then the indices yields
\begin{align*}
   S_3 &= \sum_{n \leqslant \frac{UR_1}{R}} f(n) \sum_{\frac{R}{U} < m \leqslant \min \left( \frac{R_1}{U},\frac{R_1}{n}\right)} g(m)  \, e \left( F(mn) \right) \\
   &= \sum_{n \leqslant U} f(n) \sum_{\frac{R}{U} < m \leqslant \frac{R_1}{U}} g(m)  \, e \left( F(mn) \right) + \sum_{U < n \leqslant \frac{UR_1}{R}} f(n) \sum_{\frac{R}{U} < m \leqslant \frac{R_1}{n}} g(m)  \, e \left( F(mn) \right) \\
   &= S_4 + \sum_{U < n \leqslant \frac{UR_1}{R}} f(n) \sum_{\frac{R}{U} < m \leqslant \frac{R_1}{n}} g(m)  \, e \left( F(mn) \right)
\end{align*}
and, in the $2$nd sum, since $U < n \leqslant \frac{UR_1}{R}$, we have $\frac{R}{n} < \frac{R}{U} \leqslant \frac{R_1}{n}$, so that
\begin{align*}
   S_3-S_4 &= \sum_{U < n \leqslant \frac{UR_1}{N}} f(n) \sum_{\frac{R}{n} < m \leqslant \frac{R_1}{n}} g(m)  \, e \left( F(mn) \right)-\sum_{U < n \leqslant \frac{UR_1}{R}} f(n) \sum_{\frac{R}{n} < m \leqslant \frac{R}{U}} g(m)  \, e \left( F(mn) \right) \\
   &= \sum_{n \leqslant \frac{UR_1}{R}} f(n) \sum_{\frac{R}{n} < m \leqslant \frac{R_1}{n}} g(m)  \, e \left( F(mn) \right)-S_1 - \sum_{U < n \leqslant \frac{UR_1}{R}} f(n) \sum_{\frac{R}{n} < m \leqslant \frac{R}{U}} g(m)  \, e \left( F(mn) \right)
\end{align*}
implying the asserted result.
\end{proof}
 
\section{The von Mangoldt function}

\noindent
The following result relates certain exponential sums of primes with the sum of Theorem~\ref{th:Lambda}.

\begin{prop}
\label{pro:Lambda}
Assume there exist real numbers $\alpha, \beta > 0$, $0 \leqslant \gamma < 1$ such that $2\alpha + \beta < 1$, $\alpha(\gamma-3) \leqslant \beta - \gamma$, $\alpha(\gamma + 1) + \gamma (\beta - 2) + 1 \geqslant 0$ and, for all $z \geqslant 1$ and all integers $1 < R < R_1 \leqslant 2R$ such that $R \leqslant z^{2/3}$, we have for all $\varepsilon \in \left( 0,\frac{1}{2} \right]$
\begin{equation}
   z^{-\varepsilon} \sum_{R < n \leqslant R_1} \Lambda(n) \, e \left( \frac{z}{n} \right) \ll z^\alpha R^\beta + R^\gamma. \label{eq:sum_primes}
\end{equation}
Then, for $x \geqslant e$ large
$$\sum_{n \leqslant x} \Lambda \left( \left \lfloor \frac{x}{n} \right \rfloor \right) = x \sum_{n=1}^\infty \frac{\Lambda(n)}{n(n+1)} + O_\varepsilon \left( x^{\frac{1+\alpha}{3-\beta} + \varepsilon} \right).$$
\end{prop}

\begin{proof}
By Proposition~\ref{pro:preliminary}, it suffices to estimate
$$\sum_{D < d \leqslant 2D} \Lambda(d) e \left( \frac{hx}{d+a} \right)$$
where $a \in \{0,1\}$ and for all $x^{1/3} \leqslant N < x^{1/2}$, $N < D \leqslant xN^{-1}$, $H \in \Z_{\geqslant 1}$ and $1 \leqslant h \leqslant H$. Note that $\frac{x}{d+1} = \frac{x}{d} - \frac{x}{d(d+1)}$, so that, by Abel summation, we get
\begin{align*}
   \sum_{D < d \leqslant 2D} \Lambda(d) e \left( \frac{hx}{d+1} \right) & = \sum_{D < d \leqslant 2D} \Lambda(d) \, e \left( \frac{hx}{d} \right) \, e \left( - \frac{hx}{d(d+1)} \right) \\
   & \ll \left( 1 + \frac{hx}{D^2} \right)  \max_{D \leqslant D_1 \leqslant 2D} \left| \sum_{D < d \leqslant D_1} \Lambda(d) \, e \left( \frac{hx}{d} \right) \right | \\
   & \ll (hx)^\varepsilon \left\lbrace (hx)^{1+\alpha} D^{\beta - 2} + (hx)^\alpha D^\beta + hx D^{\gamma - 2} + D^\gamma \right\rbrace 
\end{align*}
where we used \eqref{eq:sum_primes} assuming also $D \leqslant x^{2/3}$, and therefore
\begin{multline*}
   (Hx)^{-\varepsilon} \left\lbrace \frac{D}{H} + \sum_{h \leqslant H} \frac{1}{h} \sum_{a=0}^1 \left| \sum_{D < d \leqslant 2D} \Lambda(d) e \left( \frac{hx}{d+a}\right) \right | \right\rbrace \\
   \ll \frac{D}{H} + (Hx)^{1+\alpha} D^{\beta - 2} + (Hx)^\alpha D^\beta + Hx D^{\gamma - 2} + D^\gamma
\end{multline*}
provided that $H \geqslant 1$ and $N <  D \leqslant \min \left( xN^{-1},x^{2/3} \right) = xN^{-1}$, since $N \geqslant x^{1/3}$. Using Srinivasan optimization lemma on the parameter $H$, we derive
\begin{multline*}
   x^{-\varepsilon} \left( \frac{D}{H} + \sum_{h \leqslant H} \frac{1}{h} \sum_{a=0}^1 \left| \sum_{D < d \leqslant 2D} \Lambda(d) e \left( \frac{hx}{d+a}\right) \right | \right) \ll \left( x^{1+\alpha} D^{\alpha + \beta - 1} \right)^{\frac{1}{\alpha+2}} + x^{1+\alpha} D^{\beta - 2} \\
  + \left( x^\alpha D^{\alpha + \beta} \right)^{\frac{1}{\alpha+1}} + x^\alpha D^\beta + x^{1/2} D^{\frac{\gamma - 1}{2}} + x D^{\gamma - 2} +  D^\gamma.
\end{multline*}
Hence the error term of Proposition~\ref{pro:preliminary} is, for all $x^{1/3} \leqslant N < x^{1/2}$ and up to $x^{\varepsilon}$
\begin{multline*}
   \ll N + \left( x^{1+\alpha} N^{\alpha + \beta - 1} \right)^{\frac{1}{\alpha+2}} + x^{1+\alpha} N^{\beta - 2} \\
   + \left( x^{2 \alpha + \beta} N^{-\alpha - \beta} \right)^{\frac{1}{\alpha+1}} + x^{\alpha+ \beta} N^{- \beta} + x^{1/2} N^{\frac{\gamma - 1}{2}} + x N^{\gamma - 2} + \left( \frac{x}{N} \right)^\gamma.
\end{multline*}
Now choose $N = x^{\frac{1+\alpha}{3-\beta}}$. Note that the condition $2 \alpha + \beta < 1$ entails that $\frac{1+\alpha}{3-\beta} < \frac{1}{2}$, and clearly $\frac{1+\alpha}{3-\beta} > \frac{1}{3}$. We obtain that the error term is, up to $x^{\varepsilon}$
$$\ll x^{\frac{1+\alpha}{3-\beta}} + x^{\frac{\alpha^2+\alpha(3\beta-5)-\beta (2-\beta)}{(1-\alpha)(3-\beta)}} + x^{\frac{\alpha (3-2\beta) + \beta(2-\beta)}{3-\beta}} + x^{\frac{\alpha (\gamma-1)-\beta + \gamma + 2}{2(3-\beta)}} + x^{\frac{\alpha (\gamma-2) - \beta + \gamma + 1}{2(\alpha+2)(3-\beta)}} + x^{\frac{\gamma(2 - \alpha - \beta)}{3-\beta}}$$
and note that the conditions $2\alpha + \beta < 1$, $\alpha(\gamma-3) \leqslant \beta - \gamma$ and $\alpha(\gamma + 1) + \gamma (\beta - 2) + 1 \geqslant 0$ imply that the $1$st term dominates the other terms, completing the proof.
\end{proof}

\noindent
Bounds for the sum \eqref{eq:sum_primes} do already exist in the literature. For instance, in \cite[Theorem~9]{grar}, the authors proved that
$$\sum_{R < n \leqslant 2R} \Lambda(n) \, e \left( \frac{z}{n} \right) \ll z^{1/12} R^{19/24} (\log R)^{11/4}$$
provided that $1 \leqslant R \leqslant \frac{1}{5} z^{3/5}$, so that Proposition~\ref{pro:Lambda} used with $(\alpha,\beta,\gamma) = \left( \frac{1}{12},\frac{19}{24},0 \right)$ yields
$$\sum_{n \leqslant x} \Lambda \left( \left \lfloor \frac{x}{n} \right \rfloor \right) = x \sum_{n=1}^\infty \frac{\Lambda(n)}{n(n+1)} + O_\varepsilon \left( x^{26/53 + \varepsilon} \right)$$
which is slightly better than \eqref{eq:Ma_Wu}. The next result is a consequence of Proposition~\ref{pro:Vaughan_identity}.

\begin{prop}
\label{pro:Lambda_estimate}
Let $1 < R < R_1 \leqslant 2R$ be large integers and $F : \left[ 1,\infty \right) \to \left[ 0,\infty \right)$ be any map. Assume there exist $A,B >0 $ such that, for all integers $M,N \geqslant 1$
\begin{align}
   & \sum_{N < n \leqslant 2N} \underset{\frac{R}{n} < M \leqslant \frac{R_1}{n}}{\max} \left | \sum_{\frac{R}{n} < m \leqslant M} e \left( F(mn) \right) \right | \leqslant A \quad \text{for\ } N \leqslant R^{1/3} \label{eq:type_I} \\
   & \max_{\substack{M \geqslant 1 \\ MN \asymp R}}\left | \sum_{N < n \leqslant 2N} \alpha_n \sum_{M < m \leqslant 2M} \beta_m \, e \left( F(mn) \right) \right | \leqslant B \quad \text{for\ } R^{1/2} \leqslant N \leqslant 2R^{2/3} \label{eq:type_II}
\end{align}
uniformly for all complex-valued sequences $(\alpha_n)$ and $(\beta_m)$ satisfying $|\alpha_n| \leqslant 1$ and $|\beta_m| \leqslant 1$. Then
$$\sum_{R < n \leqslant R_1} \Lambda(n) \, e \left( F(n) \right) \ll \left( A + B \right) R^\varepsilon.$$
\end{prop}

\begin{proof}
Set $S_1$, $S_2$ and $S_3$ the three sums of Proposition~\ref{pro:Vaughan_identity}, where we choose $U=R^{1/3}$ and we used the normalized coefficients $\alpha_n := a_n/\log R$ and $\beta_m :=b_m 2^{- \omega(m)}$, so that $\left|\alpha_n \right| \leqslant 1$ and $\left| \beta_m \right| \leqslant 1$. Using partial summation and \eqref{eq:type_I}, we derive
$$S_1 \ll \max_{N \leqslant R^{1/3}} \sum_{N < n \leqslant 2N} \underset{\frac{R}{n} < M \leqslant \frac{R_1}{n}}{\max} \left | \sum_{\frac{R}{n} < m \leqslant M} e \left( F(mn) \right) \right | (\log R)^2 \ll A (\log R)^2.$$
We split $S_2$ into three subsums, namely
$$(\log R)^{-1} S_2 = \left( \sum_{n \leqslant R^{1/3}} + \sum_{R^{1/3}< n \leqslant R^{1/2}} + \sum_{R^{1/2} < n \leqslant R^{2/3}} \right) \alpha_n \sum_{\frac{R}{n} < m \leqslant \frac{R_1}{n}}  e \left( F(mn) \right) := S_{21}+S_{22}+S_{23}.$$
As for $S_1$, we immediately derive $S_{21} \ll A (\log R)^2$. Inverting the summations in $S_{22}$ and using \eqref{eq:type_II} yields
\begin{align*}
   S_{22} &= \sum_{R^{1/2}< n \leqslant R_1 R^{-1/3}} \ \sum_{\max \left( R^{1/3}, \frac{R}{n} \right) < m \leqslant \min \left( R^{1/2} , \frac{R_1}{n} \right) } \alpha_m  \, e \left( F(mn) \right) \log R \\
   & \ll \max_{R^{1/2}< N \leqslant 2R^{2/3}} \ \max_{\substack{M \geqslant 1 \\ MN \asymp R}}\left| \sum_{N < n \leqslant 2N} \sum_{\substack{M < m \leqslant 2M \\ m \in I_n}} \alpha_m \, e \left( F(mn) \right)\right| (\log R)^3 \\
   & \ll \max_{R^{1/2}< N \leqslant 2R^{2/3}} \ \max_{\substack{M \geqslant 1 \\ MN \asymp R}}\left| \sum_{N < n \leqslant 2N} \sum_{M < m \leqslant 2M} \alpha_m \, e \left( F(mn) \right)\right| (\log R)^4 \ll B(\log R)^4
\end{align*}
where $I_n$ is any subinterval of $\left( \frac{R}{n}, \frac{2R}{n} \right]$. Also, 
\begin{align*}   
   S_{23} & \ll \max_{R^{1/2}< N \leqslant R^{2/3}} \ \max_{\substack{M \geqslant 1 \\ MN \asymp R}}\left| \underset{R < mn \leqslant R_1}{\sum_{N < n \leqslant 2N} \alpha_n \, \sum_{M < m \leqslant 2M}} e \left( F(mn) \right)\right| (\log R)^3 \\
   & \ll \max_{R^{1/2}< N \leqslant R^{2/3}} \ \max_{\substack{M \geqslant 1 \\ MN \asymp R}}\left| \sum_{N < n \leqslant 2N} \alpha_n \, \sum_{M < m \leqslant 2M}  e \left( F(mn) \right)\right| (\log R)^4 \ll B(\log R)^4.
\end{align*}
Finally, the sum $S_3$ is splitted in two subsums, namely
$$S_3 = \left( \sum_{R^{1/3}< n \leqslant R^{1/2}} + \sum_{R^{1/2} < n \leqslant R_1 R^{-1/3}} \right) \Lambda(n) \sum_{\frac{R}{n} < m \leqslant \frac{R_1}{n}} b_m \, e \left( F(mn) \right)$$
and each of these subsums are treated similarly as for $S_{22}$ and $S_{23}$. The proof is complete.
\end{proof}

\noindent
To estimate the sum \eqref{eq:type_II}, we make appeal to the next result which is picked up from \cite[Theorem~6]{bak07}\footnote{Note that Baker's results in \cite{bak07} are stated with an extra multiplicative condition $R < mn \leqslant R_1$ with $1 < R < R_1 \leqslant 2R$, but the author removes it at the start of the proof at the cost of a factor $\log R$.}.

\begin{prop}[Baker]
\label{pro:Baker_II}
Let $X>0$, $R>1$, $M,N \geqslant 1$ such that $MN \asymp R$, $M \ll N$ and $M\ll X$, $\left(a_n\right) ,\left( b_m\right) \in \C$ such that $\left| a_n \right|,\left| b_m\right| \leqslant 1$, $\alpha ,\beta \in \R$ such that $\alpha \neq 1$, $\beta <0$ et $\alpha +\beta <2$. Set $\mathcal{L}:= \log(RX+2)$. If $\left( k, \ell \right) $ is an exponent pair, then
\begin{multline*}
\mathcal{L}^{-2} \sum_{N < n \leqslant 2N} a_n \sum_{M < m \leqslant 2M} b_m \, e\left( X\left( \frac mM\right)^{\alpha} \left( \frac nN\right) ^{\beta} \right) \\
\ll X^{1/6} \left( R^{5k+4}M^{\ell-k}\right)^{\frac 1{6 \left(k+1\right)}} + R \left( M^{-1/2} + N^{-1/4} \right).
\end{multline*}
\end{prop}

\noindent
Taking an integer $r \geqslant 1$, and applying Proposition~\ref{pro:Baker_II} with $\alpha = \beta = -r$ and $X = z(MN)^{-r}$, we derive the next estimate.

\begin{coro}
\label{cor:Baker_II}
Let $z \geqslant 1$, $r \in \Z_{\geqslant 1}$, $R>1$ such that $R \leqslant z^{\frac{2}{2r+1}}$, $M,N \geqslant 1$ such that $MN \asymp R$ and $R^{1/2} \ll N \ll R^{2/3}$, and let $\left(a_n\right) ,\left( b_m\right) \in \C$ such that $\left| a_n \right|,\left| b_m\right| \leqslant 1$. If $\left( k, \ell \right) $ is an exponent pair and $\mathcal{L}:= \log(z+2)$, then
$$\mathcal{L}^{-2} \sum_{N < n \leqslant 2N} a_n \sum_{M < m \leqslant 2M} b_m \, e\left( \frac{z}{(mn)^r} \right) \ll z^{1/6} R^{\frac{2(4-r)+k(9-2r)+\ell}{12(k+1)}} + R^{7/8}.$$
\end{coro}

\begin{proof}
Let $S_{II}$ be the sum of the left-hand side. First note that the condition $N \gg R^{1/2}$ entails that $N^2 \gg R \asymp MN$, and hence $M \ll N$. Furthermore, since $R \leqslant z^{\frac{2}{2r+1}}$ and $N \gg R^{1/2}$, we get
$$M^{r+1}N^r \asymp R^{r+1}N^{-1} \ll R^{r+1/2} \leqslant z$$
and therefore $M \ll z (MN)^{-r}$. Proposition~\ref{pro:Baker_II} may be applied with $X = z (MN)^{-r}$, yielding
$$\mathcal{L}^{-2} S_{II} \ll z^{1/6} \left( R^{k(5-r)+4-r} M^{\ell-k} \right)^{\frac{1}{6(k+1)}} + RM^{-1/2} + RN^{-1/4}$$
with $R^{k(5-r)+4-r} M^{\ell-k} \ll R^{\ell + (k+1)(4-r)} N^{k-\ell} \ll R^{\frac{1}{2}(2(4-r)+k(9-2r)+\ell)}$, $RM^{-1/2} \ll (RN)^{1/2} \ll R^{5/6}$ and $RN^{-1/4} \ll R^{7/8}$, and hence
$$\mathcal{L}^{-2} S_{II} \ll z^{1/6} R^{\frac{2(4-r)+k(9-2r)+\ell}{12(k+1)}} + R^{7/8}$$
completing the proof.
\end{proof}

\noindent
We are now in a position to establish the main estimate of this section.

\begin{prop}
\label{pro:Lambda_estimate_resultat}
Let $1 < R < R_1 \leqslant 2R$ and $z \geqslant 8$ large such that $R \leqslant z^{2/3}$. If $(k,\ell)$ is an exponent pair satisfying $k \leqslant \frac{1}{6}$ and $20k^2 + k(23-8\ell) + 2 - 7\ell > 0$, then, for all $\varepsilon \in \left( 0 , \frac{1}{2}\right) $
$$z^{-\varepsilon} \sum_{R < n \leqslant R_1} \Lambda (n) \, e \left( \frac{z}{n} \right) \ll z^{1/6} R^{\frac{7k+\ell+6}{12(k+1)}} + R^{7/8}.$$
\end{prop}

\begin{proof}
The sums \eqref{eq:type_II} are treated with Proposition~\ref{cor:Baker_II} with $r=1$. It remains to estimate the sums \eqref{eq:type_I}, for which we apply the exponent pair $(k,\ell)$, yielding
\begin{align*}
   \max_{N \leqslant R^{1/3}} \ \sum_{N < n \leqslant 2N} \underset{\frac{R}{n} < M \leqslant \frac{R_1}{n}}{\max} \left | \sum_{\frac{R}{n} < m \leqslant M} e \left( \frac{z}{mn} \right) \right | & \ll \max_{N \leqslant 2R^{1/3}} \sum_{N < n \leqslant 2N} \left\lbrace \left( \frac{z}{R} \right)^k \left( \frac{R}{n} \right)^{\ell-k} + \frac{R^2}{nz} \right\rbrace  \\
   & \ll \max_{N \leqslant R^{1/3}} \left( z^k R^{\ell-2k} N^{1-\ell+k} + R^2 z^{-1} \right) \\
   & \ll z^k R^{\frac{1+2\ell-5k}{3}} + R^2 z^{-1}.
\end{align*}
and note that $R^2 z^{-1} \leqslant R^{7/8}$ and $z^k R^{\frac{1+2\ell-5k}{3}} \leqslant z^{1/6} R^{\frac{7k+\ell+6}{12(k+1)}}$ since $k \leqslant \frac{1}{6}$ and $20k^2 + k(23-8\ell) + 2 - 7\ell > 0$. The proof is complete.
\end{proof}

\begin{proof}[Proof of Theorem~\ref{th:Lambda}]
The proof consists of a simple verification of the hypotheses of Proposition~\ref{pro:Lambda} with
$$\left( \alpha,\beta,\gamma \right) = \left( \tfrac{1}{6},\tfrac{7k+\ell+6}{12(k+1)},\tfrac{7}{8}\right).$$
The condition $3k+4\ell \geqslant 1$ ensures that $\alpha(\gamma-3) \leqslant \beta - \gamma$.
\end{proof}

\section{The Dirichlet-Piltz divisor functions}

\noindent
We first derive the analog of Proposition~\ref{pro:Lambda} for the function $\tau_r$.

\begin{prop}
\label{pro:tau_r}
Let $r \in \Z_{\geqslant 1}$ fixed, and assume there exist real numbers $\alpha, \beta > 0$ such that $2\alpha + \beta < 1$ and $4\alpha + 2 \beta > 1$ and, for all $z \geqslant 1$ and all integers $1 \leqslant R \leqslant z$, we have for all $\varepsilon \in \left( 0,\frac{1}{2} \right]$
\begin{equation}
   z^{-\varepsilon} \left\lbrace \left| \sum_{R < n \leqslant 2R} \tau_r(n) \, e \left( \frac{z}{n} \right) \right| + \left| \sum_{R < n \leqslant 2R} \tau_r(n) \, e \left( \frac{z}{n+1} \right) \right| \right\rbrace \ll z^\alpha R^\beta + R^2z^{-1} . \label{eq:sum_primes_bis}
\end{equation}
Then, for $x \geqslant e$ large
$$\sum_{n \leqslant x} \tau_r \left( \left \lfloor \frac{x}{n} \right \rfloor \right) = x \sum_{n=1}^\infty \frac{\tau_r(n)}{n(n+1)} + O_\varepsilon \left( x^{\frac{2\alpha+\beta}{2\alpha+\beta+1} + \varepsilon} \right).$$
\end{prop}

\begin{proof}
Let $x^{1/3} \leqslant N < x^{1/2}$. Using \eqref{eq:sum_primes_bis}, we derive
$$(Hx)^{-\varepsilon} \left\lbrace \frac{D}{H} + \sum_{h \leqslant H} \frac{1}{h} \sum_{a=0}^1 \left| \sum_{D < d \leqslant 2D} \tau_r(d) e \left( \frac{hx}{d+a}\right) \right | \right\rbrace \ll \frac{D}{H} + (Hx)^\alpha D^\beta + D^2 x^{-1} $$
for all $H \in \Z_{\geqslant 1}$ and all $N < D \leqslant x/N$. Using Srinivasan optimization lemma on $H$, we get
$$x^{-\varepsilon} \left( \frac{D}{H} + \sum_{h \leqslant H} \frac{1}{h} \sum_{a=0}^1 \left| \sum_{D < d \leqslant 2D} \tau_r(d) e \left( \frac{hx}{d+a}\right) \right | \right) \ll \left( x^\alpha D^{\alpha+\beta} \right)^{\frac{1}{\alpha+1}}  + x^\alpha D^\beta + D^2 x^{-1}$$
and hence the error term does of Proposition~\ref{pro:preliminary} not exceed, up to a factor $x^\varepsilon$
$$\ll N + \left( x^{2\alpha+\beta} N^{-\alpha-\beta} \right)^{\frac{1}{\alpha+1}} + x^{\alpha+\beta} N^{-\beta} + xN^{-2}.$$
Now choosing $N = x^{\frac{2\alpha+\beta}{2\alpha+\beta+1}}$ yields the asserted result plus the extra terms
$$x^{\frac{2\alpha^2+\alpha(\beta+1)+\beta}{2\alpha+\beta+1}}+x^{\frac{1-(2\alpha+\beta)}{2\alpha+\beta+1}}$$
which are absorbed by the term $x^{\frac{2\alpha+\beta}{2\alpha+\beta+1}}$ whenever $2 \alpha + \beta < 1$ and $4\alpha + 2 \beta > 1$. Also note that these two conditions ensure that $x^{1/3} \leqslant N < x^{1/2}$, as required. The proof is complete.
\end{proof}

\noindent
The treatment of the sum \eqref{eq:sum_primes_bis} rests on the next result which can be seen as an extension of the definition of the exponent pairs.

\begin{prop}
\label{pro:exp_sum_divisor}
Let $R,r \in \Z_{\geqslant 1}$ and $F \in C^\infty \left[ R, 2R \right]$, positive-valued, and such that there exists $T > 0$ such that, for all $j \in \Z_{\geqslant 0}$ and all $x \in \left[ R, 2R \right]$, we have $\left | F^{(j)} (x) \right | \asymp TR^{-j}$. If $(k,\ell)$ is an exponent pair, then
$$\sum_{R < n \leqslant 2R} \tau_r(n) \, e \left( F(n) \right) \ll T^k R^{\frac{\ell - k}{r} + 1 - \frac{1}{r}} (\log R)^r + RT^{-1} (\log R)^{r+1}.$$
\end{prop}

\begin{proof}
We use induction on $r$, the case $r=1$ being the definition of the exponent pairs. Assume the result is true for some $r \geqslant 1$. Using Corollary~\ref{cor:Dirichlet_exponential_principle} with $f = \tau_r$, $g = \mathbf{1}$ and $U = R^{\frac{r}{r+1}}$, we derive
\begin{multline*}
   \sum_{R < n \leqslant 2R} \tau_{r+1}(n) \, e \left( F(n) \right) \ll  \sum_{n \leqslant 2R^{\frac{r}{r+1}}} \tau_r(n)\left | \sum_{\frac{R}{n} < m \leqslant \frac{2R}{n}} e \left( F(mn) \right) \right | \\
   + \sum_{n \leqslant R^{\frac{1}{r+1}}} \left | \sum_{\frac{R}{n} < m \leqslant \frac{2R}{n}} \tau_r(m) \, e \left( F(mn) \right) \right | + \sum_{R^{\frac{r}{r+1}} < n \leqslant 2R^{\frac{r}{r+1}}} \tau_r(n) \left | \sum_{\frac{R}{n} < m \leqslant R^{\frac{1}{r+1}}} e \left( F(mn) \right) \right | 
\end{multline*}
and the induction hypothesis entails that 
\begin{align*}
   \sum_{R < n \leqslant 2R} \tau_{r+1}(n) \, e \left( F(n) \right) & \ll  \sum_{n \leqslant 2R^{\frac{r}{r+1}}} \tau_r(n) \left( T^k R^{\ell-k} n^{k - \ell} + \frac{R}{nT} \right) \\
   & \qquad + \sum_{n \leqslant R^{\frac{1}{r+1}}} \left( T^k R^{\frac{\ell - k}{r} + 1 - \frac{1}{r}} n^{^{- \bigl( \frac{\ell - k}{r} + 1 - \frac{1}{r} \bigr)}} (\log R)^r + \frac{R}{nT} (\log R)^{r+1} \right) \\
   & \qquad + \sum_{R^{\frac{r}{r+1}} < n \leqslant 2R^{\frac{r}{r+1}}} \tau_r(n)\left( T^k R^{\ell-k} n^{k - \ell} + \frac{R}{nT} \right)
\end{align*}
where we used in the $3$rd sum the fact that $\left( \frac{R}{n} \, , \, R^{\frac{1}{r+1}} \right] \subset \left( \frac{R}{n} \, , \, \frac{2R}{n} \right]$. The bound
$$\sum_{n \leqslant z} \frac{\tau_r(n)}{n^\alpha} \ll z^{1- \alpha} (\log z)^{r} \quad \left( 0 \leqslant \alpha \leqslant 1 \right)$$
enables us to derive 
$$\sum_{R < n \leqslant 2R} \tau_{r+1}(n) \, e \left( F(n) \right) \ll T^k R^{\frac{\ell - k}{r+1} + 1 - \frac{1}{r+1}} (\log R)^{r+1} + RT^{-1} (\log R)^{r+2}$$
completing the proof.
\end{proof}

\begin{proof}[Proof of Theorem~\ref{th:Dirichlet-Piltz}]
By Proposition~\ref{pro:exp_sum_divisor}, we obtain
$$z^{-\varepsilon} \left\lbrace \sum_{R < n \leqslant 2R} \tau_r(n) \, e \left( \frac{z}{n} \right) + \sum_{R < n \leqslant 2R} \tau_r(n) \, e \left( \frac{z}{n+1} \right) \right\rbrace \ll z^k R^{\frac{\ell - k}{r} + 1 - \frac{1}{r}-k} + R^2z^{-1}.$$
and the condition $1-\ell > k(r-1)$ ensures that $2\alpha + \beta < 1$ and $4\alpha + 2 \beta > 1$ when $\alpha = k$ and $\beta = \frac{\ell - k}{r} + 1 - \frac{1}{r}-k$. Now the result follows by applying Proposition~\ref{pro:tau_r} with these values of $\alpha$ and $\beta$.
\end{proof}

\section{The functions $\mu_2$ and $2^\omega$}

\noindent
In this section, we will make use of the function $\chi_2$ defined by
$$\chi_2(n) := \begin{cases} \mu(m), & \textrm{if\ } n = m^2 \\ 0, & \textrm{otherwise}. \end{cases}$$
We first need the following lemma.

\begin{lem}
\label{le:mu_square}
Let $1 < R < R_1 \leqslant 2R$, $z \geqslant 1$ and assume $R \leqslant z^{2/5}$. If $(k,\ell)$ is an exponent pair, then, for all $\varepsilon > 0$
$$R^{- \varepsilon} \sum_{R < n \leqslant R_1} \mu(n) \, e \left( \frac{z}{n^2} \right) \ll z^{1/6} R^{\frac{5k+\ell+4}{12(k+1)}} + z^k R^{\frac{2 \ell +1 - 8k}{3}} + R^{7/8}.$$
In particular
\begin{equation}
   R^{- \varepsilon} \sum_{R < n \leqslant R_1} \mu(n) \, e \left( \frac{z}{n^2} \right) \ll z^{1/6} R^{38/97} + R^{7/8}. \label{eq:mu_square}
\end{equation}
\end{lem}

\begin{proof}
The proof is similar to that of Proposition~\ref{pro:Lambda_estimate}, so that we only sketch the main details. Let $S_1$, $S_2$ and $S_3$ be the sums of Proposition~\ref{pro:Vaughan_identity_mu} used with $U = R^{1/3}$. By partial summation and using the exponent pair $(k,\ell)$, we get
$$R^{- \varepsilon} S_1 \ll z^k R^{\ell - 3k} \sum_{n \leqslant R^{1/3}} \frac{1}{n^{\ell-k}} + R^3 z^{-1} \ll z^k R^{\frac{2 \ell +1 - 8k}{3}} + R^3 z^{-1}.$$
Splitting $S_2$ into two subsums
$$S_2 = \left( \sum_{R^{1/3} < n \leqslant R^{1/2}} + \sum_{R^{1/2} < n \leqslant R^{2/3}} \right) a_n \sum_{\frac{R}{n} < m \leqslant \frac{R_1}{n}} \log m \, e \left( \frac{z}{(mn)^2} \right) := S_{21} + S_{22}$$ 
we interchange the summations in $S_{21}$, so that
$$S_{21} = \sum_{R^{1/2} < n  \leqslant R_1R^{-1/3}} \log n \sum_{\max \left(R^{1/3} , \frac{R}{n} \right) < m \leqslant \min \left(R^{1/2}, \frac{R_1}{n} \right) }a_m \, e \left( \frac{z}{(mn)^2} \right)$$
so that Corollary~\ref{cor:Baker_II} with $r=2$ yields
\begin{align*}
   R^{- \varepsilon }S_{21} & \ll \max_{R^{1/2} < N  \leqslant 2R^{2/3}} \ \max_{\substack{M \geqslant 1 \\ MN \asymp R}} \left| \sum_{N < n  \leqslant 2N} \mathfrak{l}_n \sum_{M < m \leqslant 2M} \alpha_m \, e \left( \frac{z}{(mn)^2} \right) \right| \\
   & \ll z^{1/6} R^{\frac{5k+\ell+4}{12(k+1)}}  + R^{7/8}
\end{align*}
where $\alpha_m := a_m 2^{-\omega(m)}$ and $\mathfrak{l}_n = \log n / \log R$, and similarly
\begin{align*}
   R^{- \varepsilon }S_{22} & \ll \max_{R^{1/2} < N  \leqslant R^{2/3}} \ \max_{\substack{M \geqslant 1 \\ MN \asymp R}} \left| \sum_{N < n  \leqslant 2N} \alpha_n \sum_{M < m \leqslant 2M} \mathfrak{l}_m \, e \left( \frac{z}{(mn)^2} \right) \right| \\
   & \ll z^{1/6} R^{\frac{5k+\ell+4}{12(k+1)}}  + R^{7/8}.
\end{align*}
The argument is similar for $S_3$, splitting the sum in two
$$S_3 = \left( \sum_{R^{1/3} < n \leqslant R^{1/2}} + \sum_{R^{1/2} < n \leqslant R_1R^{-1/3}} \right) b_n \sum_{\frac{R}{n} < m \leqslant \frac{R_1}{n}} \mu(m) \, e \left( \frac{z}{(mn)^2} \right) := S_{31} + S_{32}$$ 
and estimating $S_{31}$ and $S_{32}$ as $S_{21}$ and $S_{22}$. Also note that $R^3z^{-1} \leqslant R^{7/8}$. The last part of the proposition follows with the use of Bourgain's exponent pair $\left( \frac{13}{84} + \varepsilon, \frac{55}{84} + \varepsilon \right)$, yielding the asserted estimate with an extra term $z^{13/84} R^{5/14}$, which is absorbed by the first one.
\end{proof}

\begin{prop}
\label{pro:squarefree}
Let $1 < R < R_1 \leqslant 2R$, $z \geqslant 1$ such that $R \leqslant z^{7/10}$. Then, for all $\varepsilon > 0$
$$R^{- \varepsilon} \sum_{R < n \leqslant R_1} \mu_2(n) \, e \left( \frac{z}{n} \right) \ll z^{\frac{\np{3497}}{\np{13774}}} R^{\frac{15}{71}}.$$
\end{prop}

\begin{proof}
If $R < z^{\frac{68}{485}}$, then trivially
$$\sum_{R < n \leqslant R_1} \mu_2(n) \, e \left( \frac{z}{n} \right) \ll R \ll z^{\frac{68}{485}} \ll z^{\frac{\np{3497}}{\np{13774}}} R^{\frac{15}{71}}$$
so that we may assume $R \geqslant z^{\frac{68}{485}}$. Using Corollary~\ref{cor:Dirichlet_exponential_principle} with $f=\chi_2$ and $g = \mathbf{1}$, we derive for all $1 \leqslant U \leqslant R$
\begin{multline*}
   \sum_{R < n \leqslant R_1} \mu_2(n) \, e \left( \frac{z}{n} \right) = \sum_{n \leqslant \sqrt{\frac{UR_1}{R}}} \mu(n) \sum_{\frac{R}{n} < m \leqslant \frac{R_1}{n}} e \left( \frac{z}{mn^2}\right) \\
   + \sum_{n \leqslant \frac{R}{U}} \ \sum_{\sqrt{\frac{R}{n}} <m \sqrt{\frac{R_1}{n}}} \mu(m) \, e \left( \frac{z}{m^2n}\right) -  \sum_{\sqrt{U} < n \leqslant \sqrt{\frac{UR_1}{R}}} \mu(n) \sum_{\frac{R}{n} < m \leqslant \frac{R}{U}} e \left( \frac{z}{mn^2}\right).
\end{multline*}
For $S_1$ and $S_3$, we apply the exponent pair $(k,\ell)$ yielding
\begin{align*}
   \left| S_1 \right| & \leqslant \sum_{n \leqslant \sqrt{\frac{UR_1}{R}}} \left| \sum_{\frac{R}{n} < m \leqslant \frac{R}{n}} e \left( \frac{z}{mn^2}\right) \right| \\
   & \ll \sum_{n \leqslant \sqrt{2U}} \left\lbrace \left( \frac{z}{R}\right)^k \left( \frac{R}{n^2} \right)^{\ell-k} + \frac{R^2}{n^2z} \right\rbrace \\
   & \ll z^k R^{\ell-2k} \sum_{n \leqslant \sqrt{2U}} \frac{1}{n^{2(\ell-k)}} + R^2 z^{-1}
\end{align*}
and similarly for $S_3$. Choosing $(k,\ell) = BA\left( \frac{13}{84} + \varepsilon, \frac{55}{84} + \varepsilon \right) =  \left( \frac{55}{194} + \varepsilon, \frac{55}{97} + \varepsilon \right)$, we derive
$$R^{-\varepsilon} \left( S_1 + S_3 \right) \ll z^{55/194} \sum_{n \leqslant \sqrt{2U}} \frac{1}{n^{55/97}} + R^2 z^{-1} \ll z^{55/194} U^{21/97} + R^2 z^{-1}.$$
Now noticing that the condition $R \leqslant z^{7/10}$ entails that $\sqrt{R/n} \leqslant (z/n)^{2/5}$ for all $n \in \Z_{\geqslant 1}$, we use \eqref{eq:mu_square} for the sum $S_2$, which gives
\begin{align*}
   R^{-\varepsilon} \left| S_2 \right| & \leqslant \sum_{n \leqslant \frac{R}{U}} \left| \sum_{\sqrt{\frac{R}{n}} <m \sqrt{\frac{R_1}{n}}} \mu(m) \, e \left( \frac{z}{m^2n}\right) \right| \\
   & \ll \sum_{n \leqslant \frac{R}{U}} \left\lbrace \left( \frac{z}{n} \right)^{1/6} \left( \frac{R}{n}\right)^{19/97} + \left( \frac{R}{n}\right)^{7/16} \right\rbrace \\
   & \ll z^{1/6} R^{5/6} U^{-371/582} + RU^{-9/16}.
\end{align*}
We choose $U = \left( z^{-68} R^{485} \right)^{1/497}$. Note that the condition $R \geqslant z^{68/485}$ ensures that $1 \leqslant U \leqslant R$, and this choice of $U$ yields the asserted result with the extra terms $z^{\frac{153}{\np{1988}}} R^{\frac{\np{3587}}{7952}} + R^2z^{-1}$, which are both easily seen to be dominated by the term $z^{\np{3497}/\np{13774}} R^{15/71}$ via the hypothesis $R \leqslant z^{7/10}$.
\end{proof}

\begin{proof}[Proof of Theorem~\ref{th:squarefree}]
By Proposition~\ref{pro:preliminary}, we derive
$$\sum_{n \leqslant x} \mu_2 \left( \left \lfloor \frac{x}{n} \right \rfloor \right) = x \sum_{n=1}^\infty \frac{\mu_2(n)}{n(n+1)} + R(x)$$
with, for all $x^{1/3} \leqslant N < x^{1/2}$ and all $H \geqslant 1$
\begin{align*}
   x^{-\varepsilon} R(x) & \ll N + \max_{N < D \leqslant x/N} \left\lbrace \frac{D}{H} \right. \\
   & \left. \qquad + \sum_{h \leqslant H} \frac{1}{h} \left( \left| \sum_{D < d \leqslant 2D} \mu_2(d) \, e\left( \frac{hx}{d}\right) \right| + \left| \sum_{D < d \leqslant 2D} \mu_2(d) \, e\left( \frac{hx}{d+1}\right) \right|\right) \right\rbrace \\
   & \ll N + \max_{N < D \leqslant \min(x/N,x^{7/10})} \left\lbrace \frac{D}{H} + \sum_{h \leqslant H} \frac{1}{h} \left( 1 + \frac{hx}{D^2} \right) (hx)^{\frac{\np{3497}}{\np{13774}}} D^{\frac{15}{71}} \right\rbrace \\
   & \ll N + \max_{N < D \leqslant x/N} \left\lbrace \frac{D}{H} + (Hx)^{\frac{\np{17271}}{\np{13774}}} D^{- \frac{127}{71}} + (Hx)^{\frac{\np{3497}}{\np{13774}}} D^{\frac{15}{71}} \right\rbrace \\
   & \ll N + \max_{N < D \leqslant x/N} \left\lbrace \left( x^{\np{17271}} D^{- \np{7367}} \right)^{\frac{1}{\np{31045}}}  + x^{\frac{\np{17271}}{\np{13774}}} D^{- \frac{127}{71}} +  \left( x^{\np{3497}} D^{\np{6407}} \right)^{\frac{1}{\np{17271}}} + x^{\frac{\np{3497}}{\np{13774}}} D^{\frac{15}{71}} \right\rbrace 
\end{align*}
where we used Srinivasan optimization lemma in the last line, and the fact that $\min(\frac{x}{N},x^{7/10}) = \frac{x}{N}$. Therefore
$$x^{-\varepsilon} R(x) \ll N + \left( x^{\np{17271}} N^{- \np{7367}} \right)^{\frac{1}{\np{31045}}} + x^{\frac{\np{17271}}{\np{13774}}} N^{- \frac{127}{71}} +  \left( x^{\np{9904}} N^{-\np{6407}} \right)^{\frac{1}{\np{17271}}} + x^{\frac{\np{6407}}{\np{13774}}} N^{-\frac{15}{71}}$$
and choosing $N = x^{\frac{\np{1919}}{\np{4268}}}$ yields the asserted bound.
\end{proof}

\noindent
The next result is the analog of Proposition~\ref{pro:squarefree} for the function $2^\omega$. The proof is much simpler.

\begin{prop}
\label{pro:unitary}
Let $1 < R < R_1 \leqslant 2R$, $z \geqslant 1$, $(k, \ell)$ be an exponent pair and assume $R \leqslant z^{\frac{2(k+1)}{3(k+1)-\ell}}$. Then, for all $\varepsilon > 0$
$$R^{- \varepsilon} \sum_{R < n \leqslant R_1} 2^{\omega(n)} \, e \left( \frac{z}{n} \right) \ll z^{k} R^{\frac{1+\ell-3k}{2}}.$$
\end{prop}

\begin{proof}
Using Corollary~\ref{cor:Dirichlet_exponential_principle} with $f=\chi_2$, $g = \tau$ and $U=R$, we derive
$$\sum_{R < n \leqslant R_1} 2^{\omega(n)} \, e \left( \frac{z}{n} \right) = \sum_{n \leqslant \sqrt{R_1}} \mu(n) \sum_{\frac{R}{n^2} < m \leqslant \frac{R_1}{n^2}} \tau(m) \, e \left( \frac{z}{mn^2} \right)$$
and Proposition~\ref{pro:exp_sum_divisor} with $r=2$ yields
\begin{align*}
   R^{- \varepsilon} \sum_{R < n \leqslant R_1} 2^{\omega(n)} \, e \left( \frac{z}{n} \right) & \ll \sum_{n \leqslant \sqrt{R_1}} \left\lbrace \left( \frac{z}{R}\right)^k \left(\frac{R}{n^2} \right)^{\frac{1+\ell-k}{2}} + \frac{R^2}{n^2 z} \right\rbrace \\
   & \ll z^{k} R^{\frac{1+\ell-3k}{2}} \sum_{n \leqslant \sqrt{R_1}} \frac{1}{n^{1+\ell-k}} + R^2 z^{-1} \ll z^{k} R^{\frac{1+\ell-3k}{2}}
\end{align*}
since $1+\ell-k \geqslant 1$ and the term $R^2z^{-1}$ is absorbed by the term $z^{k} R^{\frac{1+\ell-3k}{2}}$ with the help of the hypothesis $R \leqslant z^{\frac{2(k+1)}{3(k+1)-\ell}}$.
\end{proof}

\begin{proof}[Proof of Theorem~\ref{th:unitary_divisors}]
First note that, if $x^{1/3} \leqslant N < x^{1/2}$ and if $(k,\ell)$ is an exponent pair, then
$$x^{\frac{2(k+1)}{3(k+1)-\ell}} \geqslant x^{2/3} \geqslant \frac{x}{N}.$$
Now we proceed as in the proof of Theorem~\ref{th:squarefree} above, using Propositions~\ref{pro:preliminary} and~\ref{pro:unitary} to derive
$$\sum_{n \leqslant x} 2^{\omega\left( \left \lfloor x/n \right \rfloor \right)} = x \sum_{n=1}^\infty \frac{2^{\omega(n)}}{n(n+1)} + R(x)$$
with, for all $x^{1/3} \leqslant N < x^{1/2}$ and all $H \geqslant 1$
\begin{align*}
   x^{-\varepsilon} R(x) & \ll N + \max_{N < D \leqslant x/N} \left\lbrace \frac{D}{H} \right. \\
   & \left. \qquad + \sum_{h \leqslant H} \frac{1}{h} \left( \left| \sum_{D < d \leqslant 2D} 2^{\omega(d)} \, e\left( \frac{hx}{d}\right) \right| + \left| \sum_{D < d \leqslant 2D} 2^{\omega(d)} \, e\left( \frac{hx}{d+1}\right) \right|\right) \right\rbrace \\
   & \ll N + \max_{N < D \leqslant \min\Bigl (x/N \, , \, x^{\frac{2(k+1)}{3(k+1)-\ell}} \Bigr )} \left\lbrace \frac{D}{H} + \sum_{h \leqslant H} \frac{1}{h} \left( 1 + \frac{hx}{D^2} \right) (hx)^k D^{\frac{1+\ell-3k}{2}} \right\rbrace \\
   & \ll N + \max_{N < D \leqslant x/N} \left\lbrace \frac{D}{H} + (Hx)^{1+k} D^{- \frac{3+3k-\ell}{2}} + (Hx)^k D^{\frac{1+\ell-3k}{2}} \right\rbrace \\
   & \ll N + \max_{N < D \leqslant x/N} \left\lbrace x^{\frac{k+1}{k+2}} D^{- \frac{1+k-\ell}{2(k+2)}} + x^{1+k} D^{- \frac{3+3k-\ell}{2}} + x^{\frac{k}{k+1}} D^{\frac{1+\ell-k}{2(k+1)}} + x^k D^{\frac{1+\ell-3k}{2}} \right\rbrace 
\end{align*}
where we used Srinivasan optimization lemma again, and hence
$$x^{-\varepsilon} R(x) \ll N + x^{\frac{k+1}{k+2}} N^{- \frac{1+k-\ell}{2(k+2)}} + x^{1+k} N^{- \frac{3+3k-\ell}{2}} + x^{\frac{k+\ell+1}{2(k+1)}} N^{- \frac{1+\ell-k}{2(k+1)}}+ x^{\frac{1+\ell-k}{2}} N^{- \frac{1+\ell-3k}{2}}.$$
Choose $N = x^{\frac{2(k+1)}{3k-\ell+5}}$. Note that that the condition $k+\ell < 1$ ensures that $x^{1/3} \leqslant N < x^{1/2}$. We then get the asserted result with the following two extra terms
$$x^{\frac{5k^2+8k+(3-\ell)(1+\ell)}{2(k+1)(3k-\ell+5)}} + x^{\frac{3k^2+(3+2k-\ell)(1+\ell)}{2(3k-\ell+5)}}$$
which are absorbed by the first term via the condition $k+\ell < 1$. The last part of the Theorem is derived with the exponent pair $(k,\ell) = \left( \frac{13}{84} + \varepsilon, \frac{55}{84} + \varepsilon \right)$.
\end{proof}

\section{The function $\omega$}

\begin{prop}
\label{pro:omega}
Let $1 < R < R_1 \leqslant 2R$ and $z \geqslant 1$ such that $R \leqslant z^{26/41}$. Then, for all $\varepsilon > 0$
$$R^{- \varepsilon} \sum_{R < n \leqslant R_1} \omega(n) \, e \left( \frac{z}{n} \right) \ll z^{1/6} R^{128/195}.$$
\end{prop}

\begin{proof}
Using Corollary~\ref{cor:Dirichlet_exponential_principle}, we derive for $1 \leqslant U \leqslant R$
\begin{multline*}
   \sum_{R < n \leqslant R_1} \omega(n) \, e \left( \frac{z}{n} \right) = \sum_{p \leqslant \frac{UR_1}{R}} \ \sum_{\frac{R}{p} < n \leqslant \frac{R_1}{p}} e \left( \frac{z}{np} \right) + \sum_{n \leqslant \frac{R}{U}} \ \sum_{\frac{R}{n} < p \leqslant \frac{R_1}{n}} e \left( \frac{z}{np} \right) \\
   - \sum_{U < p \leqslant \frac{UR_1}{R}} \, \sum_{\frac{R}{p} < n \leqslant \frac{R}{U}} e \left( \frac{z}{np} \right) := S_1 + S_2 + S_3.
\end{multline*}
With the exponent pair $(k,\ell)$, we get
\begin{align*}
   \left| S_1 \right| & \leqslant \sum_{p \leqslant \frac{UR_1}{R}} \left| \sum_{\frac{R}{p} < n \leqslant \frac{R_1}{p}} e \left( \frac{z}{np} \right)  \right| \ll \sum_{p \leqslant 2U} \left( z^k p^{k-\ell} R^{\ell-2k} + \frac{R^2}{pz} \right)  \\
   & \ll U^{k+1-\ell} z^k R^{\ell - 2k}  + R^2 z^{-1} \log \log (e^2R)
\end{align*}
and a similar bound holds for $S_3$. By partial summation and Proposition~\ref{pro:Lambda_estimate_resultat}, we derive
\begin{align*}
   S_2 & \ll \sum_{n \leqslant \frac{R}{U}} \ \max_{\frac{R}{n} \leqslant M \leqslant \frac{R_1}{n}} \left| \sum_{\frac{R}{n} < m \leqslant M} \Lambda(m) e \left( \frac{z}{mn} \right) + M^{1/2} \right| \\
   & \ll \sum_{n \leqslant \frac{R}{U}} \left\lbrace \left( \frac{z}{n}\right)^{1/6} \left( \frac{R}{n} \right)^{\frac{7k+\ell+6}{12(k+1)}} + \left( \frac{R}{n} \right)^{7/8} + \left( \frac{R}{n} \right)^{1/2} \right\rbrace \\
   & \ll \sum_{n \leqslant \frac{R}{U}} \left\lbrace z^{1/6} R^{\frac{7k+\ell+6}{12(k+1)}} n^{-\frac{9k+\ell+8}{12(k+1)}} + \left( \frac{R}{n} \right)^{7/8} \right\rbrace \\
   & \ll z^{1/6} R^{5/6} U^{-\frac{3k-\ell+4}{12(k+1)}} + RU^{-1/8}.
\end{align*}
Optimizing in $U$ and taking $(k,\ell) = \left( \frac{1}{6},\frac{2}{3} \right)$ yields the asserted result plus the extra terms
$$z^{1/30} R^{13/15} + R^{7/8} + R^2z^{-1}$$
which are all absorbed by the term $z^{1/6} R^{128/195}$ via the condition $R \leqslant z^{26/41}$.
\end{proof}

\begin{proof}[Proof of Theorem~\ref{th:omega}]
As in the previous theorems, using Propositions~\ref{pro:preliminary} and~\ref{pro:omega}, we get
$$\sum_{n \leqslant x} \omega \left( \left \lfloor x/n \right \rfloor \right) = x \sum_{n=1}^\infty \frac{\omega(n)}{n(n+1)} + R(x)$$
with, for all $x^{1/3} \leqslant N < x^{1/2}$ and all $H \geqslant 1$
\begin{align*}
   x^{-\varepsilon} R(x) & \ll N + \max_{N < D \leqslant x/N} \left\lbrace \frac{D}{H} \right. \\
   & \left. \qquad + \sum_{h \leqslant H} \frac{1}{h} \left( \left| \sum_{D < d \leqslant 2D} \omega(d) \, e\left( \frac{hx}{d}\right) \right| + \left| \sum_{D < d \leqslant 2D} \omega(d) \, e\left( \frac{hx}{d+1}\right) \right|\right) \right\rbrace \\
   & \ll N + \max_{N < D \leqslant \min\Bigl (x/N \, , \, x^{26/41} \Bigr )} \left\lbrace \frac{D}{H} + \sum_{h \leqslant H} \frac{1}{h} \left( 1 + \frac{hx}{D^2} \right) (hx)^{1/6} D^{128/195} \right\rbrace. 
\end{align*}
Assume $x^{15/41} \leqslant N < x^{1/2}$. We derive
\begin{align*}
   x^{-\varepsilon} R(x) & \ll N + \max_{N < D \leqslant x/N} \left\lbrace \frac{D}{H} + (Hx)^{1/6} D^{128/195} + (Hx)^{7/6} D^{-262/195} \right\rbrace \\
   & \ll N + \max_{N < D \leqslant x/N} \left\lbrace x^{1/7} D^{321/455} + x^{7/13} D^{-69/845} + x^{1/6} D^{128/195} + x^{7/6} D^{-262/195} \right\rbrace \\
   & \ll N + \left( x^{386} N^{-321} \right)^{1/455} + x^{7/13} N^{-69/845} + x^{107/130} N^{-128/195} + x^{7/6} N^{-262/195} \\
   & \ll x^{455/914} + x^{841/\np{1690}} + x^{193/388} + x^{451/910} + x^{321/646} + x^{193/390} + x^{15/41} \ll x^{455/914}
\end{align*}
as claimed.
\end{proof}

\end{document}